\documentclass[10pt]{amsart} 
\usepackage{latexsym,enumerate,amscd,}
\usepackage[abbrev]{amsrefs} 

\usepackage[nocolor]{sseq}
\usepackage[textwidth=1.2in, textsize=small]{todonotes}
\usepackage{tikz, tikz-cd}
\usetikzlibrary{arrows}
\usetikzlibrary{matrix}
\usetikzlibrary{shapes}
\usetikzlibrary{calc}
\usepgflibrary{shapes.geometric}

\theoremstyle{plain} 
\newtheorem{Theorem}{Theorem}[section]

\newtheorem{Corollary}[Theorem]{Corollary} 
\newtheorem{Lemma}[Theorem]{Lemma}

\theoremstyle{definition} 
\newtheorem*{Definition}{Definition} 
\newtheorem*{Remark}{Remark} 

\newtheorem*{Acknowledgements}{Acknowledgements}
\newtheorem*{Strong Atiyah conjecture}{Strong Atiyah conjecture}
\newtheorem*{Strong Atiyah conjecture(Algebraic Version)}{Strong Atiyah conjecture (Algebraic Version)}

\newcommand{\gG}{\Gamma}
\newcommand\BZ{\mathbb{Z}} 

\newcommand*{\xMin}{0}%
\newcommand*{\xMax}{3}%
\newcommand*{\yMin}{0}%
\newcommand*{\yMax}{-3}%

\newcommand*{\zMin}{6}%
\newcommand*{\zMax}{10}%

\newcommand*{\sMin}{2}%
\newcommand*{\sMax}{7}%
\newcommand*{\tMin}{-7}%
\newcommand*{\tMax}{-10}%

\DeclareMathOperator{\Lk}{Lk}
\DeclareMathOperator{\Sym}{Sym}
\DeclareMathOperator{\tr}{tr}
\DeclareMathOperator{\lcm}{lcm} 

\begin{document}

\title{The Strong Atiyah conjecture for virtually cocompact special groups}

\author{Kevin Schreve}
\address{ Dept. of Mathematical Sciences\\
  University of Wisconsin\\
  Milwaukee, WI  USA 53201}
\email{kschreve@uwm.edu}

\date{}



\begin{abstract}
We provide new conditions for the Strong Atiyah conjecture to lift to finite group extensions. In particular, we show these conditions hold for cocompact special groups so the Strong Atiyah conjecture holds for virtually cocompact special groups.
\end{abstract}
\maketitle

\section{Introduction}
\label{intro}

The motivation for this work comes from the following question of Atiyah \cite[p.~72]{a76} about $L^2$-Betti numbers of a manifold $Y$ with a cocompact proper $G$-action (we slightly change the notation):
\begin{quotation} A priori the numbers $\ell^{2}b^{k}(Y;G)$ are real. Give examples
where they are not integral and even perhaps irrational.
\end{quotation}
Of course, non-integral examples are well-known and can be easily constructed for any group with torsion.  
Irrational examples were first constructed by T. Austin~\cite{a09}. Additional examples appear in \cites{g10,psz10}. However, all these examples involve groups with unbounded torsion.  In fact,   for the class of groups with bounded torsion the following conjecture is still open.
\begin{Strong Atiyah conjecture}
	For any proper cocompact $G$-CW complex $Y$ 
	\[
		\lcm G \cdot \ell^{2}b^{k}(Y;G) \in \BZ,
	\]
	where $\lcm G$ denotes the least common multiple of the orders of finite subgroups of $G$.
\end{Strong Atiyah conjecture}
	(If $G$ contains arbitrarily large finite subgroups, the conjecture is vacuous.)
	
	Since  $\ell^{2}b^{k}(Y;G)$ is the von Neumann dimension of the kernel of the Laplacian operator $\Delta \in M_k(\BZ G)$, the above conjecture is implied by, and in fact equivalent to, the following:
	
	\begin{Strong Atiyah conjecture(Algebraic Version)}
	Let $A \in M_n(\BZ G)$. Then 
	\[
		\lcm G \cdot \dim_G(\ker A) :=   \lcm G \cdot \tr_G(\pi_{\ker A}) := \lcm G \cdot \sum_{i = 1}^n \langle \pi_{\ker A} \delta_1e_i, \delta_1e_i \rangle \in \BZ.
	\]
\end{Strong Atiyah conjecture(Algebraic Version)}
Here, $\pi_{\ker A}$ is the orthogonal projection from $\ell_2(G)^n$ onto the kernel, and $\delta_1e_i := (0, \dots, 0, \delta_1, 0, \dots, 0) \in \ell_2(G)^n$ is the standard basis element with non-zero entry the characteristic function of the identity element of G. This conjecture has also been generalized to matrices $A \in M_n(KG)$ for $K$ a subfield of $\mathbb{C}$.

The conjecture is known for a large class of groups. It easily follows from the multiplicative properties of $\ell^{2}b^{k}$ that the conjecture holds for all finite groups. It has also been shown for free groups and elementary amenable groups \cite{l93}, residually torsion-free elementary amenable groups \cite{s00}, and  right-angled Artin and Coxeter groups  \cite{los12}. Unfortunately, the conjecture does not behave well when passing to subgroups or taking finite extensions. A special case is when $G$ is torsion-free, in which case the Strong Atiyah conjecture for $G$ implies the conjecture for all subgroups. In the case of group extensions, a fundamental theorem was proven by Linnell and Schick in \cite{ls07}, where they gave conditions for the conjecture to lift to all elementary amenable extensions. Their main conditions, stated loosely here, are the following:

Let $\hat G^p$ denote the pro-$p$ completion of a discrete group $G$ for a prime $p$.   $G$ is called \textit{cohomologically $p$-complete} if the canonical homomorphism $G \rightarrow \hat G^p$ induces an isomorphism on cohomology with $\BZ/p\BZ$ coefficients. $G$ is \textit{cohomologically complete} if it is cohomologically $p$-complete for all $p$.

 $G$ has \textit{enough torsion-free quotients} if every map from $G$ to a finite $p$-group  factors through a torsion-free elementary amenable group.

In this paper, we modify enough torsion-free quotients to the \textit{factorization property}, which requires that any map from $G$ to a \emph{finite} group factors through a torsion-free elementary amenable group. This strengthening lets us relax the cohomological condition on $G$ to being a good group in the sense of Serre, which requires $G$ to have the same cohomology as its \emph{profinite} completion. Given these conditions on a group, we conclude the same lifting of the conjecture as Linnell and Schick. 

\begin{Theorem}\label{Atiyah1}	
	Suppose $1\to H\to G\to Q\to 1$ is an exact sequence of groups where $H$ satisfies the Strong Atiyah conjecture and $Q$ is elementary amenable. Suppose $H$ has a finite classifying space, is a good group, and  has the factorization property.
Then $G$ satisfies the Strong Atiyah conjecture. 
\end{Theorem}

A \textit{cocompact special group} is the fundamental group of a compact special cube complex as defined by Haglund and Wise in \cite{hw}.
Our motivation here is that while cocompact special groups have enough torsion-free quotients, there are known examples of cocompact special groups that are not cohomologically complete. However, Lorensen proved in \cite{MR2561762} that right-angled Artin groups are good, so it easily follows from the work of Haglund and Wise that compact special groups are good. Proving cocompact special groups have the factorization property gives us:

\begin{Theorem}\label{vcs1}
The Strong Atiyah conjecture holds for virtually cocompact special groups.
\end{Theorem}

The paper is organized as follows: In Sections 2 and 3 we investigate goodness and the factorization property. The most important results for us are that both conditions pass to finite index subgroups and retracts, while also holding for right-angled Artin groups. In Sections 4 and 5 we prove Theorem \ref{Atiyah1}. Section 6 is devoted to proving the Strong Atiyah conjecture for virtually cocompact special groups, and in Section 7 we look at an example of a link group that is good and not cohomologically complete.

\begin{Acknowledgements} I would like to thank Prof. Peter Linnell for sending a proof that finite index subgroups of right-angled Artin groups have enough torsion-free quotients. I would also like to thank my advisor Prof. Boris Okun for all of his help and insight throughout this paper. I would also like to thank an anonymous referee for helpful suggestions and for pointing out an error in an earlier proof of Theorem \ref{grfact}.
\end{Acknowledgements} 

\section{Good groups}

In this section we record a few of the facts known about good groups; in particular,  a result of Lorensen in \cite{MR2561762} that right-angled Artin groups are good.

 \begin{Definition} 
For a group $G$, let $\hat G = \varprojlim_{[G:H] < \infty} G/H,$ where the inverse limit is taken over the set of finite quotients of $G$. $\hat G$ is  the \emph{profinite completion} of $G$. For every $G$, there is a canonical homomorphism $i: G \rightarrow \hat G$ which sends $g \in G$ to the cosets $gH$.
 A  group $G$ is called \textit{good} if the homomorphism $$H^\ast(\hat G, M) = \varinjlim_{[G:H] < \infty} H^\ast(G/H, M) \overset{i^\ast}\rightarrow H^\ast(G, M)$$ is an isomorphism for every finite $G$-module $M$.
 \end{Definition}

\begin{Lemma}[{\cite[Exercise 2(a)]{s97}}] 
Suppose we have an exact sequence $1 \rightarrow H \rightarrow G \rightarrow Q \rightarrow 1$  with $Q$ finite and $H$ finitely generated. Then the induced sequence of profinite completions $1 \rightarrow \hat H \rightarrow \hat G \rightarrow Q \rightarrow 1$ is exact.
\end{Lemma}

\begin{Lemma}[{\cite[Exercise 2(b)]{s97}}] \label{good1}
Let $1 \rightarrow H \rightarrow G \rightarrow K \rightarrow 1$ be an extension with $H, K$ good and $H^\ast(H,M)$ finite for every G-module M. Then $G$ is good.
\end{Lemma}

\begin{Lemma}[{\cite[Lemma 3.2]{gjz07}}] \label{good2}
Suppose $G$ is a good group. If $H$  is commensurable to $G$ then $H$  is good. In particular, goodness passes to finite index subgroups.
\end{Lemma}

\begin{Lemma}\label{good3}
Suppose $G$ is a good group and $H$  is a retract of $G$. Then $H$  is good.
\end{Lemma}

\begin{proof}
The cohomology of $G$ and $\hat G$ are both contravariant functors. Therefore, the canonical map  $H^n(\hat H, M) \rightarrow H^n(H, M)$ is a direct summand of the canonical map $H^n(\hat G, M) \rightarrow H^n(G, M)$. Since the latter is an isomorphism, so is the former.
\end{proof}

\begin{Definition}
Let $\gG$ be a finite simplicial graph with vertex set $S$. There is a \emph{right-angled Artin group $A_\gG$} associated to $\gG$ which is generated by $s_i \in S$, and has relations $s_is_j = s_js_i$ if and only if $s_i$ and $s_j$ span an edge of $\gG$.
\end{Definition}

For example, complete graphs produce free abelian groups, while graphs with no edges produce free groups.  It is also natural to consider the flag complex generated by $\gG$; we shall abuse notation and denote this $\gG$ as well. The flag complex is formed by gluing an $n$-simplex onto any complete subgraph of $\gG$ with $n$ vertices. Therefore, $n$-simplices of the flag complex $\gG$ correspond to $n$ pairwise commuting generators of the Artin group.

\begin{Theorem}[{\cite[Theorem  3.15]{MR2561762}}]\label{good} All right-angled Artin groups are good.
\end{Theorem}
\begin{proof}
Let $A_\gG$ be the right-angled Artin group based on a flag complex $\gG,$ and choose $s \in \gG$. Then $A_\gG$ decomposes as the HNN extension $A_{\gG - s} \ast_{A_{\Lk(s)}}$, where $\Lk(s)$ denotes the link of $s$ in $\gG$. Assume $A_{\gG - s}$ and ${A_{\Lk(s)}}$ are good by induction on the number of generators. On the level of completions $\widehat A_{\gG - s}$ and $\widehat A_{\Lk(s)}$ inject into $\widehat A_\gG$, as both subgroups are retracts of the latter (this fact is shown in \cite{MR2561762}.) Therefore, using the Mayer--Vietoris Sequence for HNN extensions and the Five Lemma we conclude that $A_\gG$ is good:

$$ 
	\begin{CD}
		\dots\leftarrow H^n(A_{\Lk(s)},M) @<i<<  H^n(A_{\gG - s},M) @<j<< H^n(A_\gG, M) \leftarrow \dots \\
		@V\hat iVV @V\hat iVV @V\hat iVV\\
		\dots\leftarrow H^n(\hat A_{\Lk(s)},M) @<i<<  H^n(\hat A_{\gG - s},M) @<j<< H^n(\hat A_\gG, M) \leftarrow \dots \\
			\end{CD}
	$$

\end{proof}

\begin{Definition} If $\gG$ is a finite simplicial graph with vertex set $S$, suppose we are given a family of groups $(G_s)_{s \in S}$. The \emph{graph product} $G_\gG$ is defined as the quotient of the free product of the $(G_s)_{s \in S}$ by the normal subgroup generated by the commutators of the form $[g_s,g_t]$ with $g_s \in G_s, g_t \in G_t$, where $s$ and $t$ span an edge of $\gG$.  
\end{Definition}

A graph product is a natural generalization of many interesting groups. For example, a right-angled Artin group is a graph product with each vertex group $\BZ$, and a right-angled Coxeter group is a graph product with each vertex group $\BZ_2$. A similar proof 
shows that graph products of good groups are good, and graph products of cohomologically complete groups are cohomologically complete.

\section{The factorization property}

In this section we introduce the factorization property and show that it holds for right-angled Artin groups.
Recall that the class of elementary amenable groups is the smallest class of groups which contains all finite and abelian groups and is closed under taking subgroups, quotients, extensions, and directed unions. For example, every solvable group is elementary amenable, and every elementary amenable group is amenable.

\begin{Definition} A group $H$ has \emph{the factorization property} if any map from $H$ to a finite group factors through a torsion-free elementary amenable group. Equivalently, for any finite index normal subgroup $K \trianglelefteq H$ there is $U \trianglelefteq H$ such that $U \trianglelefteq K$ and $H/U$ is torsion-free elementary amenable.
\end{Definition}

 \begin{Lemma}\label{l:fin}
A finite index subgroup of a group with the factorization property has the factorization property. \end{Lemma}
\begin{proof}
Let $K < G$ be a finite index subgroup of a group with the factorization property, and let $f: K \to P$ be
a map to a finite group. The left action of $G$ on the set of cosets $G/\ker f$ gives a map $g:G\to \Sym (G/\ker f)$ to a finite permutation group, which by assumption factors through a torsion-free elementary amenable group $M$.
Denote by $h$ the map $G \to M$.
Since $\ker f$ is the stabilizer of the trivial coset, $\ker g \subset \ker f$,  and therefore $f$ factors through  $g(K)$. Therefore, $f$ factors through  $h(K)$ which is torsion-free elementary amenable as a subgroup of $M$.

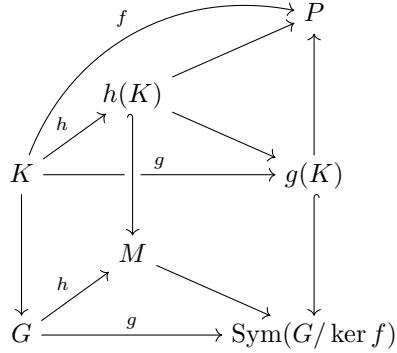
\begin{center}
\begin{tikzcd}[row sep=scriptsize, column sep=scriptsize]
&& P \\
&  h(K) \ar{ur}   \ar{dr} \\
K \ar{rr}{g}\ar{ur}{h} \ar{dd} \ar[bend left = 40]{uurr}{f} && g(K)\ar{uu} \ar[hook]{dd} \\
 & M  \ar{dr} \ar[hookleftarrow,  crossing over]{uu}\\
G \ar{rr}{g}  \ar{ur}{h} && \Sym (G/\ker f) \\
\end{tikzcd}
\end{center}
\end{proof}

\begin{Lemma}\label{I:ret}

Suppose $G$ has the factorization property. Then any retract of $G$ has the factorization property.
\end{Lemma}
\begin{proof}
Let $p: G \rightarrow H$ be a retraction, and let $f: H \rightarrow P$ be a map from $H$  onto a finite group. By assumption on $G, f \circ p$ factors through a torsion-free elementary amenable group $M$. By pre-composing with $i: H \rightarrow G,$ we see $f$ factors though a subgroup of $M$, and hence $H$ has the factorization property.

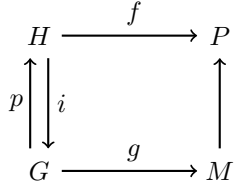
\begin{center}
\begin{tikzpicture}[thick, scale = 0.6]

\node at (0,0) {$H$};
\node at (4,0) {$P$};
\node at (0,-3) {$G$};
\node at (4,-3) {$M$};

\node at (2.1, .5) {$f$};
\node at (-.5,-1.5) {$p$};
\node at (.5, -1.5) {$i$};
\node at (2.1, -2.6) {$g$};

\draw [->] (.2,-.5) -- (.2,-2.5);
\draw [<-] (-.2,-.5) -- (-.2,-2.5) ;

\draw [<-] (4,-.5) -- (4,-2.5);
\draw [->] (.5,0) -- (3.5,0);
\draw [->] (.5,-3) -- (3.5,-3);

\end{tikzpicture}

 \end{center}
 \end{proof}

\begin{Lemma}\label{l:far} If $E$ is a free group and $F$ is a normal subgroup, then $E/[F,F]$ is torsion-free.
\end{Lemma}
\begin{proof}

The statement is a corollary of  the proof of Lemma 5 in \cite{Farkas}. Our argument comes from a comment by Agol \cite{MO80657}. Note that $[F,F]$ is a normal subgroup of E as $[F,F]$ is characteristic in F. Choose $g \in E-F$. The subgroup $G$ = $\langle F,g \rangle$ generated by $F$ and $g$ has $G / F$ cyclic. Consider the normal series $G \trianglerighteq [G,G] \trianglerighteq [F,F]$. As $G/F$ is abelian, $[G,G] \trianglelefteq F;$ therefore $[G,G]/[F,F]$ is free abelian as a subgroup of $F/[F,F]$. Therefore, $G/[F,F]$ is torsion-free by the sequence $$1 \rightarrow [G,G]/[F,F] \rightarrow G/[F,F] \rightarrow G/[G,G] \rightarrow 1.$$

\end{proof}

 \begin{Lemma}\label{l:ext}
Extensions of free groups by  torsion-free elementary amenable groups have the factorization property.
\end{Lemma}
\begin{proof}
Let $1 \to E \to H \to M \to 1$ be such an extension with $E$ free and $M$ torsion-free elementary amenable. Let $f: H \to P$ be a map to a finite group. 
Let $F = E \cap \ker f$.
Then $E/ F$ is finite, and $E / [F,F] $ is torsion-free by Lemma \ref{l:far} and elementary amenable by the exact sequence $$1 \rightarrow F/ [F,F] \rightarrow E / [F,F] \rightarrow E/F \rightarrow 1.$$ 
Now $f$ factors through $H/ [F,F]$ 
which is torsion-free elementary amenable by the exact sequence $$1 \rightarrow E/[F,F] \rightarrow H/[F,F] \rightarrow H/E \rightarrow 1.$$
\end{proof}

\begin{Theorem}\label{grfact} Graph products of groups with the factorization property have the factorization property.
\end{Theorem}

\begin{proof}

Let  $ f:G_\gG \to P$ be a map to a finite group. By induction on the number of vertices of $\gG$, assume the restriction of  $f$ to $G_{\gG - s}$ factors through a torsion-free elementary amenable group $N$. By hypothesis, the restriction of $f$ to $G_s$  factors through a torsion-free elementary amenable group $K$. If $G_\gG = G_{\gG - s} \times G_s$, we can factor $f$ through the product $N \times K$.
Otherwise, if $G_{\Lk(s)}$ denotes the graph product based on the link of $s \in  \Gamma,$ $G_\gG$ splits as an amalgamated  product $$G_\gG =G_{\gG - s} \ast_{G_{\Lk(s)}} (G_{\Lk(s)} \times G_s).$$
Let $L$ be the image of $G_{\Lk(s)}$ in $N$.  The above factorizations induce a factorization of $f$ through $ N \ast_L (L \times K)$. By mapping $N$ and $L \times K$ into $N \times K$, we get a map  $ N \ast_L (L \times K) \rightarrow N \times K$. Since the kernel of this map is free and $N \times K$ is torsion-free elementary amenable, we can apply Lemma \ref{l:ext} to factor $f$ through a torsion-free elementary amenable group.

\end{proof}

 $\BZ$ trivially has the factorization property, so as a corollary we have: 

\begin{Corollary}\label{I:factor} Right-angled Artin groups have the factorization property.
\end{Corollary}

\section{Proof of Theorem \ref{Atiyah1}}

Our strategy for Theorem \ref{Atiyah1} follows that of Linnell and Schick in \cite{ls07}. We quickly give an overview of the proof, which will take up the next two sections.

 Given a finite extension $1 \rightarrow H \rightarrow G \overset{f}\rightarrow Q \rightarrow 1$ with $H$ torsion-free, Linnell and Schick give conditions on $H$ so that $f$ factors through an elementary amenable group $G/U$ with $\lcm(G/U) = \lcm(G)$. This is enough to show the conjecture for $G$, which we record as a key lemma.

\begin{Lemma}[{\cite[Theorem  2.6]{ls07}}] \label{tfea}  Let $1\rightarrow H \rightarrow G \rightarrow M \rightarrow 1$ be an extension where $H$ is torsion-free, satisfies the Strong Atiyah conjecture, $M$ is elementary amenable and $\lcm(M) = \lcm(G)$. Then $G$ satisfies the Strong Atiyah conjecture.
\end{Lemma}

The factorization property can be thought of as a condition on $H$ which guarantees a lot of torsion-free elementary amenable quotients. Consider one of these quotients, i.e. a subgroup $U \trianglelefteq H$ with $H/U$ torsion-free elementary amenable. It is natural to make each $U$ into a normal subgroup of $G$, denoted by $U^G$, and consider the quotient $G/U^G$. While $G/U^G$ is always elementary amenable, it is tricky to control its torsion. Using the factorization property, we show that if all the quotients are bad in the sense that they have a lot of torsion, this implies a splitting  of $Q$ to the profinite completion $\hat G$. With some work,  we can use goodness of $G$ to guarantee that at least one of the quotients $G/U^G$ has $\lcm(G/U) = \lcm(G)$.

We now try to make the above ideas precise and prove our main theorem.
\begin{Definition}
For a CW-complex $Y$, let $\pi_S^\ast(Y)$ and $\tilde \pi_S^*(Y)$ denote the \emph{unreduced and reduced stable cohomotopy groups} of $Y$ respectively. If $Y$ is a point, then the stable cohomotopy of $Y$ is by definition isomorphic to the stable homotopy groups of spheres --- in this case we shorten $\pi_S^*(Y)$ to $ \pi_S^*$. If $G$ is a discrete group, $\tilde \pi_S^\ast(G)$ is defined to be the reduced stable cohomotopy of the classifying space $BG$. If $\hat G$ is the profinite completion of $G$, define $$\tilde \pi_S^*(\hat G) :=  \varinjlim_{[G:N] < \infty} \tilde \pi_S^*(G/N).$$ More details can be found in Section 4.4 of  \cite{ls07}.
\end{Definition}

The next theorem is a profinite version of Theorem 4.27 in \cite{ls07}. It is the main tool used to control torsion in these finite extensions.

\begin{Theorem}\label{annoying}
Suppose we have an exact sequence $$1 \rightarrow H \rightarrow G \rightarrow Q \rightarrow 1$$ where Q is a finite $p$-group and $H$  is good and has a finite classifying space. Assume the induced exact sequence of profinite completions $$1 \rightarrow \hat H \rightarrow \hat G \rightarrow Q \rightarrow 1$$ splits. Then the original sequence also splits.
\end{Theorem}

The idea behind Theorem \ref{annoying} is that a splitting on the exact sequence of profinite completions should induce an injection of the degree zero reduced stable cohomotopy $i^\ast: \tilde \pi_S^0(\hat G) \rightarrow \tilde \pi_S^0(G)$. We actually prove something slightly weaker than this; but we end up with a description of what an element of the kernel should look like. 
Since $Q$ splits to $\hat G$, $\tilde \pi_S^0(Q) \rightarrow \tilde \pi_S^0(\hat G)$ is injective, and with some work we conclude an injection  $\tilde \pi_S^0(Q) \rightarrow \tilde \pi_S^0(G)$.
By the following result of \cite{ls07}, where it is attributed to A. Adem, this guarantees a splitting $Q \rightarrow G$.

\begin{Theorem} [{\cite[Theorem  4.28]{ls07}}] Let $H$ be a discrete group with finite cohomological dimension. Suppose we have an extension $$1 \rightarrow H  \rightarrow G \overset{f}\rightarrow  Q \rightarrow 1$$ where $Q$ is a finite $p$-group. The extension above splits if and only if the epimorphism $G \rightarrow Q$ induces an injection $\tilde \pi_S^0(Q) \rightarrow \tilde \pi_S^0(G)$.
\end{Theorem}

Therefore, Theorem \ref{annoying} follows from the following technical lemma, which we prove in the next section.

\begin{Lemma}\label{annoying2}
Suppose that we have an exact sequence $$1 \rightarrow H \rightarrow G \rightarrow Q \rightarrow 1$$ where Q is a finite $p$-group and $H$  is good and has a finite classifying space. Assume the induced exact sequence of profinite completions $$1 \rightarrow \hat H \rightarrow \hat G \rightarrow Q \rightarrow 1$$ splits. Then the induced map $i^\ast: \tilde \pi_S^0(Q) \rightarrow \tilde \pi_S^0(G)$ is injective.
\end{Lemma}

Before proving our main theorem, we need the following four lemmas --- these correspond to Lemmas  4.10, 4.52, 4.54 and 4.59 in \cite{ls07}.

\begin{Lemma}\label{nice}
Suppose $H$ is finite index and normal in $G$, and $\mathcal{U} = \{U\}$ is a collection of normal subgroups of $H$ such that $H/U$ is torsion-free elementary amenable for each $U \in \mathcal{U}$. Let $\mathcal{U}^G = \{U^G\}$ be the corresponding collection of normal subgroups of $G$, where $U^G = \cap_{g \in G} \hspace{1mm} gUg^{-1}$. Then 
\begin{enumerate}[(i)] 

\item This is a finite intersection.
\item $H/U^G$ is torsion-free elementary amenable.
\item $G/U^G$ is elementary amenable.
\end{enumerate}

\end{Lemma}

\begin{proof}
\begin{enumerate}[(i)]

\item This follows from $H$ being finite index in $G$ and $U$ being normal in $H$. 

\item $U^G$ is the kernel of the map $$H \rightarrow H/U \times H/g_1Ug^{-1} \times \dots \times H/g_nUg_n^{-1}$$ where the range is assumed to be torsion-free elementary amenable.

\item We have the exact sequence $$1 \rightarrow H/U^G \rightarrow G/U^G \rightarrow G/H \rightarrow 1.$$

\end{enumerate}
\end{proof}

\begin{Lemma}\label{split1}
Let $Q$ be a finite $p$-group in the exact sequence $$1 \rightarrow H \rightarrow G \xrightarrow{\pi} Q \rightarrow 1.$$ Assume that among all normal finite index subgroups of $G$, there is a cofinal system $U_i \trianglelefteq G$ with $U_i \subset H$, such that for each $i,$ the homomorphism $\pi_i$ in $$G \xrightarrow{p_i} G/U_i \xrightarrow{\pi_i} Q$$ has a split $s_i: Q \rightarrow G/U_i$. Then the profinite completion map $\hat \pi: \hat G \rightarrow Q$ has a split $Q \rightarrow \hat G$.
\end{Lemma}

\begin{proof}
The proof in \cite{ls07} works identically in this case. The idea is that for each $q \in Q,$ we choose elements $g(q,i) \in G$ with $p_i(g(q,i)) = s_i(q) \in G/U_i$. Since $\hat G$ is compact, each sequence $p_i(g(q,i))$ has a convergent subsequence, and since $Q$ is finite, we can assume there is one subsequence with $p_i(g(q,i)) \rightarrow g(q)$ for each $q$. The splitting is then defined as
$$s: Q \rightarrow \hat G, q \mapsto g(q).$$ 
 \end{proof}

\begin{Lemma}\label{collection}
Suppose $H$ is finitely generated and has the factorization property. Then there exists a collection $\mathcal{U}$ of subgroups $U \trianglelefteq H$ such that every finite index subgroup of $H$ contains a subgroup in $\mathcal{U}$, if $U \in \mathcal{U}$ then $H/U$ is torsion-free elementary amenable, and if  $U \in \mathcal{U}, V \in \mathcal{U},$ then $U \cap V \in \mathcal{U}$.
\end{Lemma}

\begin{proof}
Let $K_n$ be the intersection of all subgroups of H of index $\le n$. Since $K_n$ is finite index in $H$ and $H$ has the factorization property, there is a subgroup $V_n$ contained in $K_n$ such that $H/V_n$ is torsion-free elementary amenable.  Letting $U_n= \cap_{i=1}^n V_i$ and $\mathcal{U} = \{U_n\}$ satisfies the above conditions, and by the proof of Lemma \ref{nice}$(ii)$, $H/U_n$ is torsion-free elementary amenable for each $n$.

\end{proof}

\begin{Lemma}\label{split2}
Let $Q$ be a finite $p$-group and let $1 \rightarrow H \rightarrow G \xrightarrow{\pi} Q \rightarrow 1$ be an exact sequence of groups. Assume $H$ is finitely generated and has the factorization property. Let $\mathcal{U}$ be a collection of normal subgroups of $H$ as in Lemma \ref{collection}, and define $\mathcal{U}^G$ as above. If each $G/U^G$ contains a subgroup of order $p^k$, then there is a subgroup $Q_0 < Q$ of order $p^k$ splitting back to $\hat G_0 \le \hat G$, where $G_0 = \pi^{-1}(Q_0)$.
\end{Lemma}

\begin{proof}
By an easy argument in \cite{ls07}, there exists $Q_0$ of order $p^k$ splitting back to $G/U^G$ for all elements in $\mathcal{U}^G$. If $Q_0$ splits to $G/U^G$, then it splits to the quotient $G/J$ for all $U^G \trianglelefteq J \trianglelefteq H$. Amongst the normal finite index subgroups of $G$ (or $G_0$), those contained in $H$ form a cofinal collection. Since $H$ has the factorization property, each finite index $K \trianglelefteq H$ contains $U^G \in \mathcal{U}^G,$ so $Q_0$ splits to each finite quotient $G/K$ (or $G_0/K)$. By Lemma \ref{split1}, this implies that $Q_0$ splits to $\hat G_0$.

\end{proof}

\begin{proof}[Proof of Theorem \ref{Atiyah1}]	
	
Using \cite[Lemma 2.4, Corollary 2.7]{ls07}, we only need to prove the case of $Q$ being a finite $p$-group.
 Let $\mathcal{U}$ be a collection of normal subgroups of $H$ as in Lemma \ref{collection}, and let $\mathcal{U}^G$ as above. If $\lcm(G/U^G) = \lcm(G)$ for any $U^G \in \mathcal{U}^G,$ we would be done by Lemma \ref{tfea}, Lemma \ref{nice}, and the extension $$1 \rightarrow U^G \rightarrow G \rightarrow G/U^G \rightarrow 1.$$ 

Since $H$ and $H/U^G$ are torsion-free and $Q$ is a finite $p$-group, $\lcm(G)$ and 
$\lcm (G/U^G)$ are powers of $p$.
Now, suppose each of the above groups $G/U^G$ had a torsion subgroup of order $p^k$. By Lemma \ref{split2}, there is a subgroup $Q_0$ in $Q$ of order $p^k$ and a splitting $Q \rightarrow \hat G_0$. Theorem \ref{annoying} now implies $Q_0$ splits to $G_0,$ which implies $\lcm(G) \ge p^k$. Therefore, there exists a subgroup $U^G$ such that $\lcm(G/U^G) = \lcm(G)$.
\end{proof}

\begin{Remark} If $G$ is torsion free, there is a quicker proof that does not require Theorem \ref{annoying}. The idea is the same as above, but a splitting $Q_0 \rightarrow \hat G_0$ is an easy contradiction as goodness implies $H^\ast (\hat G_0, \BZ/p\BZ) = H^\ast (G_0, \BZ/p\BZ)$ is zero above some dimension, while $H^\ast(Q, \BZ/p\BZ)$ is not. In this case, we can also relax the assumption of $H$ having a finite classifying space to being finitely generated and having finite cohomological dimension.
\end{Remark}

\begin{Remark}
Schick has shown that the methods used in \cite{ls07} apply to the Baum--Connes conjecture with coefficients. Unsurprisingly, our results can be applied in the same way. We will just state our theorem as the proof is identical to his in \cite{sch07}.

\begin{Theorem}\label{bc5}
Suppose $1\to H\to G\to Q\to 1$ is an exact sequence of groups, where $H$ and $Q$ satisfy the Baum--Connes conjecture conjecture with coefficients. Suppose $H$ has a finite classifying space, is a good group, and has the factorization property. Suppose $G$ is torsion-free. Then $G$ satisfies the Baum--Connes Conjecture with coefficients.
\end{Theorem}
\end{Remark}

\section{Proof of  Lemma  \ref{annoying2}}

In this section, we complete the proof of Theorem \ref{Atiyah1}.
\begin{proof}[Proof of Lemma  \ref{annoying2}]

We use the same spectral sequence argument as in \cite{ls07}. Recall that we have an Atiyah--Hirzebruch spectral sequence for $\tilde \pi_S^{s+t} (G)$ with $$E_2^{s,t}(G) = \tilde H^s(G, \pi_S^t),$$ and a corresponding spectral sequence for $\tilde \pi_S^{s+t} (\hat G)$ with $$E_2^{s,t}(\hat G) = \tilde H^s(\hat G, \pi_S^t).$$ This second spectral sequence is defined to be the direct limit of the spectral sequences for the finite quotients of $G$. 

Since $G$ has torsion,  our classifying space for $G$ is infinite dimensional, so we need to be careful about the convergence of the Atiyah--Hirzebruch spectral sequence. If $X$ is a connected CW-complex and $X^{(k)}$ its $k$-skeleton, let $F_k^*(X)$ denote the kernel of the map $j^\ast : \tilde \pi_S^*(X) \rightarrow \tilde \pi_S^*(X^{(k)})$ induced by the inclusion $j: X^{(k)} \rightarrow X$. Assume without loss of generality that $X^{(0)}$ is a point; in this case $F_0^\ast(X) =  \tilde \pi_S^*(X)$ as $\tilde \pi_S^\ast$ is trivial. We say the above spectral sequence converges to $\tilde \pi_S^{s+t}(X)$ if $$E_\infty^{s,t}(X) \cong F_s^{s+t}(X)/F_{s+1}^{s+t}(X)\hspace{10mm} \forall s \ge 0,t \in \BZ.$$

Now we compare our two spectral sequences. Recall that $\pi_S^\ast$ is trivial if $\ast > 0$ and finite if $\ast < 0$. Therefore, these are both fourth quadrant spectral sequences with finite coefficients for $t < 0$. 
Since $G$ is good by Lemma \ref{good1}, we get an isomorphism on  $E_2^{s,t}$ terms except for $s > 0, t = 0$ (for $s = t = 0$ both terms are trivial.) This implies isomorphic $E_\infty^{s,t}$ terms for $s+t \le 0$ as in this range there is no interaction with the possibly non-isomorphic terms, as indicated below in Figure 1.

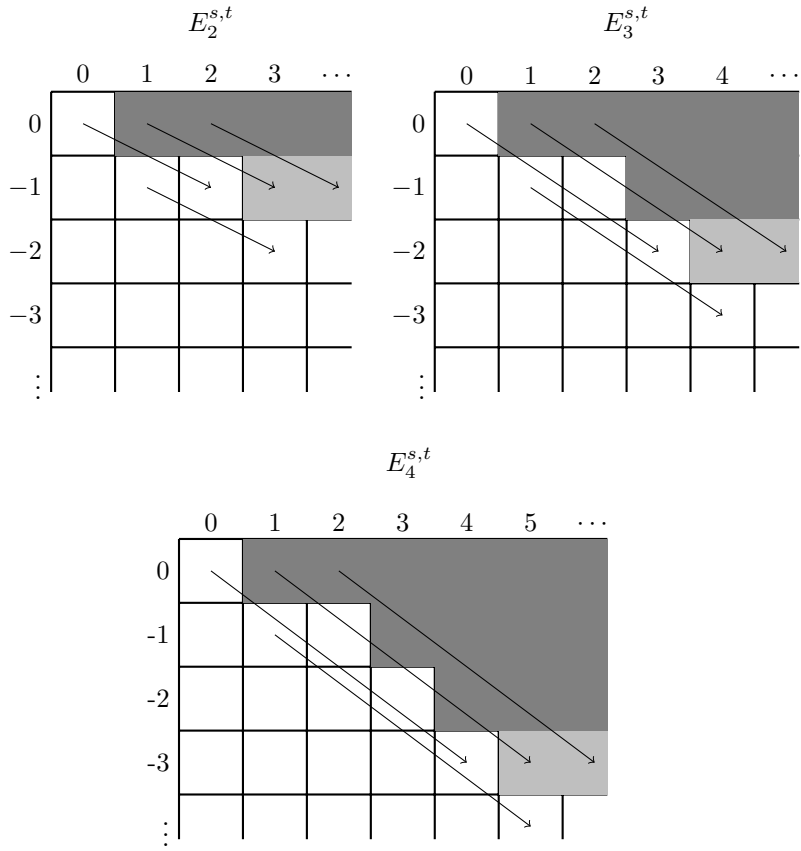
\begin{figure}[ht]
\centering
\begin{tikzpicture}[scale = .85]

\begin{scope}[xshift=2cm, yshift=-2cm]
 \foreach \i in {\xMin,...,\xMax} {
        \node [black,above] at (\i+0.5,\yMin) {$\i$};
    }
     \node [black,above] at (\xMax +1.5,\yMin) {$\cdots$};

  \foreach \i in {\yMin,...,\yMax} {
         \node [black,left] at (\xMin,\i-0.5) {$\i$};
    }
 \node [black,left] at (\xMin,\yMax-1.5) {$\vdots$};    

 \draw [step=1.0,black, thick, yshift=\yMax cm - 2cm] (0,.3) grid +(\xMax+1.7, -\yMax+1.7);

 \path [fill =  gray] (1,0) rectangle (4.7,-1);
 \path [fill =  lightgray] (3,-1) rectangle (4.7,-2);

\node [black,above] at (2.5,.7) {$E_2^{s,t}$};
\node [black,above] at (9,.7) {$E_3^{s,t}$};
\draw[<-]  (2.5,-1.5) -- (.5,-.5);
\draw[<-]  (3.5,-2.5) -- (1.5,-1.5);
\draw[<-]  (4.5,-1.5) -- (2.5,-.5);
\draw[<-]  (3.5,-1.5) -- (1.5,-.5);

\foreach \i [evaluate=\i as \j using \i - 6] in {\zMin,...,\zMax} {
         \node [black,left] at (\i +.75,\yMin+.25) {\pgfmathtruncatemacro{\yet}{\j}\yet};
  
    }

  \foreach \i in {\yMin,...,\yMax} {
         \node [black,left] at (\zMin,\i-0.5) {$\i$};
    }
  \node [black,above] at (\zMax +1.5,\yMin) {$\cdots$};   
 \node [black,left] at (\zMin,\yMax-1.5) {$\vdots$};    

 \draw [step=1.0,black, thick, yshift=\yMax cm - 2cm] (6,.3) grid +(5.7, -\yMax+1.7);

 \path [fill =  gray] (7,0) rectangle (11.7,-1);
 \path [fill =  gray] (9,-1) rectangle (11.7,-2);
 \path [fill =  lightgray] (10,-2) rectangle (11.7,-3);

\draw[<-]  (9.5,-2.5) -- (6.5,-0.5);
\draw[<-]  (10.5,-2.5) -- (7.5,-0.5);
\draw[<-]  (10.5,-3.5) -- (7.5,-1.5);
\draw[<-]  (11.5,-2.5) -- (8.5,-.5);

\foreach \i [evaluate=\i as \j using \i - 2] in {\sMin,...,\sMax} {
         \node [black,left] at (\i +.75,\tMin+.25) {\pgfmathtruncatemacro{\yet}{\j}\yet};
  
    }

\node [black,above] at (5.6,-6.2) {$E_4^{s,t}$};

  \foreach \i [evaluate=\i as \j using \i + 7]  in {\tMin,...,\tMax} {
         \node [black,left] at (\sMin,\i -.5) {\pgfmathtruncatemacro{\yet}{\j}\yet};
    }
  \node [black,above] at (\sMax + 1.5,\tMin) {$\cdots$};   
 \node [black,left] at (\sMin,\tMax- 1.5) {$\vdots$};    

 \draw [step=1.0,black, thick, yshift=\yMax cm - 2cm] (2,-6.7) grid +(6.7, -\yMax+1.7);

\path [fill =  gray] (3,-7) rectangle (8.7,-8);
\path [fill =  gray] (5,-8) rectangle (8.7,-9);
\path [fill =  gray] (6,-9) rectangle (8.7,-10);
\path [fill =  lightgray] (7,-10) rectangle (8.7,-11);

\draw[<-]  (6.5,-10.5) -- (2.5,-7.5);
\draw[<-]  (7.5,-11.5) -- (3.5,-8.5);

\draw[<-]  (7.5,-10.5) -- (3.5,-7.5);
\draw[<-]  (8.5,-10.5) -- (4.5,-7.5);
\end{scope}

 \end{tikzpicture}
\caption{The shaded regions indicates terms that interact with the non-isomorphic top row $(s > 0)$. These terms have no effect on the lower diagonal of the $E_\infty$ sheet. }
\end{figure}

In particular, we have an isomorphism on the diagonal $E_\infty^{k,-k}(\hat G) \rightarrow E_\infty^{k,-k}(G)$ for all $k$.  We now use the convergence results of \cite{ls07}, where it is shown in Propositions 4.43 and 4.47 that $F_{k}^0(\hat G)/F_{k+1}^0(\hat G)$ injects into $E_\infty^{k,-k}(\hat G)$ and $F_{k}^0(G)/F_{k+1}^0(G)$ is isomorphic to $E_\infty^{k,-k}(G)$. 

(The key idea behind these  results is that the $E_2$-terms of both spectral sequences are finite; for $E_2^{s,t}(G)$ this is trivial, while for $E_2^{s,t}( \hat G)$ it follows from $G$ being good. Since these terms are finite, they must stabilize after a finite number of sheets of the spectral sequence, and with some care convergence follows.)

The convergence results imply an injection: $$F_{k}^0(\hat G)/F_{k+1}^0(\hat G) \rightarrow  E_\infty^{k,-k}(\hat G) \cong E_\infty^{k,-k}( G) \rightarrow F_{k}^0(G)/F_{k+1}^0(G).$$ We now have a commutative diagram for all $k$: 

$$\begin{tikzcd}[row sep=small, column sep=small]
0  \ar{r} & F_{k}^0(\hat G)/F_{k+1}^0(\hat G) \ar{r} \ar{d}{i_k} & \tilde \pi_S^0(\hat G)/F_{k+1}^0(\hat G) \ar{r} \ar{d}{\Phi_{k+1}} & \tilde \pi_S^0(\hat G)/F_k^0(\hat G) \ar{r}  \ar{d}{\Phi_k} & 0\\
0  \ar{r} &  F_{k}^0(G)/F_{k+1}^0( G) \ar{r}  & \tilde \pi_S^0(G)/F_{k+1}^0(G) \ar{r} &             \tilde \pi_S^0(G)/F_k^0(G) \ar{r} & 0\\
\end{tikzcd}$$

The injectivity of $\Phi_{k+1}$ follows from the injectivity of $i_k$ and $\Phi_k$ by a diagram chase. By induction beginning with $k = 0$, we have injections for each $k \ge 0$: $$\tilde \pi_S^0(\hat G)/F_k^0(\hat G) \rightarrow \tilde \pi_S^0(G)/F_k^0(G).$$

We also have the following commutative diagram, where $s$ is by assumption split-injective:

$$\begin{tikzcd}[row sep=scriptsize, column sep=scriptsize]
\tilde \pi_S^0(Q) \ar{rr} \ar{d}{s}&& \tilde \pi_S^0(Q)/F_k^0(Q) \ar{d}{s}\\
\tilde \pi_S^0(\hat G) \ar{d} \ar{rr} && \tilde \pi_S^0(\hat G)/F_k^0(\hat G)\ar{d}{\Phi_k} \\
\tilde \pi_S^0(G) \ar{rr} && \tilde \pi_S^0(G)/F_k^0(G)\\
\end{tikzcd}$$

We have shown that $\Phi_k$ is injective for each $k \ge 0$. This implies any element of the kernel of $\tilde \pi_S^0(Q) \rightarrow \tilde \pi_S^0(G)$  would be contained in $F^0_k(Q)$ for each $k \in \mathbb{N}$.
However, it was shown in Proposition 4.40 of \cite{ls07} that $\cap_k \hspace{1mm} F^0_k(Q)$ is trivial, which implies $\tilde \pi_S^0(Q) \rightarrow \tilde \pi_S^0(G)$ is injective.
\end{proof}

\section{Virtually cocompact special Groups}
A \emph{cube complex} $K$ is a polyhedral complex where each cell is isometric to the Euclidean cube $[0,1]^n$. A cube complex is \emph{non-positively curved} if the link of each vertex is a flag simplicial complex. If a cube complex is non-positively curved and simply-connected, then the natural piecewise Euclidean metric on $K$ is CAT(0).

If $K$ is a non-positively curved cube complex, Haglund and Wise in \cite{hw} made the remarkable observation that $\pi_1(K)$ injects into a right-angled Artin group $A_\Gamma$ if and only if the hyperplanes of $K$ are embedded in $K$ and avoid certain configurations. They define such a cube complex to be \emph{special}.  A group $H$ is then \emph{special} if is the fundamental group of a special cube complex $K$, and \emph{cocompact special} if $K$ is compact. 

\begin{Theorem} [{\cite[Theorem  2]{los12}}]  Right-angled Artin groups satisfy the Strong Atiyah conjecture. Consequently, special groups satisfy the Strong Atiyah conjecture.
\end{Theorem}

The next theorem is an immediate consequence of \cite[Proposition 6.5]{hw}, which inspired much of the work in this paper.
\begin{Theorem}\label{retract}
Let $H$ be a cocompact special group, and $i: H \rightarrow A_\gG$ the embedding constructed by Haglund and Wise. Then there is a finite index subgroup $K$ of $A_\gG$ which retracts onto $H$.
\end{Theorem}

Obviously, cocompact special groups have finite classifying spaces. Since goodness and the factorization property pass to finite index subgroups and retracts (Lemmas \ref{good2}, \ref{good3}, \ref{l:fin}, \ref{I:ret}) and right-angled Artin groups are good and have the factorization property (Theorem \ref{good} and Corollary \ref{I:factor}), we conclude:

\begin{Corollary}\label{gf}
Cocompact special groups have finite classifying spaces,  are good, and have the factorization property.
\end{Corollary}

We can now prove our second theorem from the introduction, which gives a new class of groups satisfying the Strong Atiyah conjecture.

\begin{proof}[Proof of Theorem \ref{vcs1}]

Let $G$ be a virtually cocompact special group. As finite index subgroups of cocompact special groups are cocompact special \cite{hw}, we can assume that $G$ has a normal cocompact special subgroup $H$. By Corollary \ref{gf}, Theorem \ref{Atiyah1} applies to $H$. 
\end{proof}

\begin{Remark}
The class of virtually cocompact special groups has been shown to be amazingly large, highlighted by  Agol's recent proof of the Virtually Compact Special Theorem. 

\begin{Theorem}[{\cite[Theorem  1.1]{a12}}]  Let $G$ be a word-hyperbolic group that acts properly and compactly on a $CAT(0)$ cube complex $X$. Then there exists a finite index subgroup $G'$ of $G$ such that $X/G'$ is special.
\end{Theorem}

Recent breakthroughs in $3$-manifold theory and Theorem \ref{vcs1} imply that the Strong Atiyah conjecture holds for all fundamental groups of closed (or with toroidal boundary) hyperbolic 3-manifolds.
 \end{Remark}

Coxeter groups were cubulated in \cite{nr02}, and shown to be virtually special in \cite{hw1}. It is also known that the cubulation is cocompact whenever the group does not contain a Euclidean triangle Coxeter subgroup. 

\begin{Corollary}
 The Strong Atiyah conjecture holds for all Coxeter groups which do not contain a Euclidean triangle Coxeter subgroup. In particular, it holds for all word-hyperbolic Coxeter groups.
 \end{Corollary}

\section{Knot and Link Complements}

In this section we give an example from \cite{br12}  which illustrates the advantage of our conditions over those of \cite{ls07}.
 Recall that a \emph{link group} $G$ is the fundamental group of the complement $M$ of a tamely embedded link in $S^3$. Similarly, a \emph{knot group} is the fundamental group of the complement of a tame knot in $S^3$. In \cite{ls07}, it was shown that all knot groups are cohomologically complete. The idea is that the cohomology of knot groups $G$ is well-known by Alexander duality: $$H^n(G, \BZ_p) = \begin{cases}
\BZ_p, & n = 0,1 \\
0, & n \ge 2
\end{cases}
$$

Therefore, the map to the abelianization $G \rightarrow G/[G,G] \cong \BZ$ induces an isomorphism on all homology groups. We now use a theorem of Stallings in \cite{st65} to conclude that for any prime $p$, the pro-$p$ completion $\hat G^p \cong \hat {\BZ}^p$, which implies completeness of $G$.

\begin{Remark}
Wilton and Zalesskii prove in \cite{wz09} that if $M$ is a closed, irreducible, orientable 3-manifold, goodness of $\pi_1(M)$ follows from goodness of all fundamental groups of pieces of the JSJ decomposition of $M$. We have seen that fundamental groups of closed hyperbolic $3$-manifolds, or those with toroidal boundary, are good. Since fundamental groups of Seifert-fibered spaces are good, it follows that all knot groups are good.
\end{Remark}

Blomer, Linnell and Schick in \cite{bls08} also show cohomological completeness for certain link complements called \emph{primitive link groups} (this is a combinatorial condition on the linking diagram.) This lets them conclude:

\begin{Theorem}[{\cite[Theorem  1.4]{bls08}}] \label{bls} 
Let $H$ be a knot group or a primitive link group. If $H$ satisfies the Strong Atiyah conjecture, then every elementary amenable extension of $H$ satisfies the Strong Atiyah conjecture.
\end{Theorem}

On the other hand, Bridson and Reid in \cite{br12} reverse the argument in \cite{ls07} to construct link groups that are not cohomologically complete. These examples are \emph{homology boundary links}, which have the property that the corresponding link group $G$ surjects onto $F_2$. An example of theirs is shown below.
\\

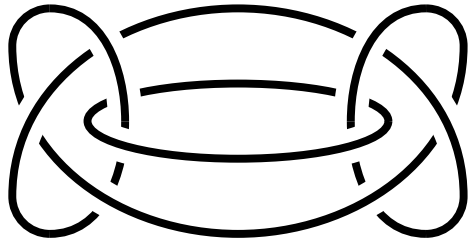
\begin{center}
\tikzset{knot/.style={white,double distance=3pt,line width=5pt, double=black}}
\begin{tikzpicture}[thick, scale = 0.5]

\draw[knot] (-4,0) to[out=-90,in=0] (-6,-3);
\draw[knot] (2,0) to[out=-90,in=180] (4,-3);

\draw[knot] (3,0) arc (0:360:4 and 1);

\draw[knot] (-6,3) to[out=180,in=90]  (-7,2) to[out=-90,in=180] (-1,-3) to[out=0,in=-90] (5,2)  to[out=90,in=0] (4,3) ;
\draw[knot] (-6,-3) to[out=180,in=-90]  (-7,-2) to[out=90,in=180] (-1,3) to[out=0,in=90] (5,-2) to[out=-90,in=0] (4,-3);

\draw[knot] (-4,0) to[out=90,in=0] (-6,3);
\draw[knot] (2,0) to[out=90,in=180] (4,3);

\end{tikzpicture}

\end{center}

Again using Alexander duality and Stallings Theorem, Bridson and Reid show that for any prime $p$, $\hat G^p \cong \hat F_2^p$, which has homology concentrated in dimension 1 as $F_2$ is cohomologically complete. However, link groups with more than 2 components have nontrivial second homology groups, which contradicts completeness in this case. 

Clearly, Theorem \ref{bls} does not apply to this homology boundary link example. However, the example shown in the figure (and in general `most'  link complements) are hyperbolic. We have shown above that the link group is good and has the factorization property, so Theorem \ref{Atiyah1} applies and the Atiyah conjecture holds for elementary amenable extensions of this group.

\end{document}